\documentclass[12pt]{article}
\usepackage{cite}
 \usepackage[english]{babel}
\usepackage{amsfonts,amsmath,amsxtra,amsthm,amssymb,latexsym}
\newcommand{\bR}{{\mathbb R}}

\newcommand{\bC}{{\mathbb C}}

\newcommand{\kC}{{\mathcal C}}
\newcommand{\kN}{{\mathcal N}}
\newcommand{\kH}{{\mathcal H}}
\newcommand{\gotH}{{\mathfrak H}}

\newcommand{\gG}{{\Gamma}}
\newcommand{\bed}{\begin{displaymath}}
\newcommand{\eed}{\end{displaymath}}

\newcommand{\gga}{{\gamma}}

\newcommand{\gT}{{\Theta}}
\newcommand{\gL}{{\Lambda}}

\newcommand\diag{\mathop{diag}}
\newcommand{\ba}{\begin{array}}
\newcommand{\ea}{\end{array}}
 \newcommand\dom{\operatorname{dom}}

\newcommand{\slim}{\,\mbox{\rm s-}\hspace{-2pt} \lim}
\newcommand{\ran}{{\mathrm{ran\,}}}
\newcommand{\ack}{\section*{Acknowledgments}}
\newtheorem{theorem}{Theorem}[section]

 \newtheorem{corollary}[theorem]{Corollary}
 
 \newtheorem{proposition}[theorem]{Proposition}
 \theoremstyle{definition}
 \newtheorem{definition}[theorem]{Definition}
 \theoremstyle{remark}
 \newtheorem{remark}[theorem]{Remark}
 
 \numberwithin{equation}{section}

\def\cC{{\mathcal C}}
\def\cH{{\mathcal H}}
\def\mul{{\rm mul\,}}
\def\Ext{{\rm Ext\,}}
\DeclareMathOperator{\imm}{Im}
 \DeclareMathOperator{\Span}{span}

\begin{author}
{Nataly Goloshchapova}
\end{author}

\begin{title}
{  On the Multi-Dimensional  Schr\"odinger Operators with  Point
Interactions}
\end{title}
\date{}
\begin{document}
\maketitle

\begin{abstract}
 We study   two- and three-dimensional matrix
Schr\"odinger operators with  $m\in \mathbb N$  point
 interactions. Using the technique   of  boundary  triplets and the corresponding Weyl
functions, we  complete  and  generalize  the  results obtained by
the other  authors  in this field.

 For  instance,  we  parametrize all
self-adjoint extensions of the initial minimal symmetric
Schr\"odinger operator by abstract boundary conditions and
characterize their spectra. Particularly,   we find a sufficient
condition  in terms of distances  and  intensities  for  the
self-adjoint extension $H_{\alpha,X}^{(3)}$ to have $m'$ negative
eigenvalues, i.e., $\kappa_-(H_{\alpha,X}^{(3)})=m'\le m$. We also
give an explicit description of self-adjoint nonnegative
extensions.
\end{abstract}\qquad\\
\textbf{Mathematics  Subject  Classification (2000)}. Primary
47A10, 47B25; Secondary  47A40.\\
\textbf{Key  words}. Schr\"odinger operator, point interactions,
self-adjoint extensions, nonnegative extensions, scattering
matrix.
\section{Introduction}

 Multi-dimensional  Schr\"odinger operators with point  interactions have been intensively studied in the three last
decades (see
\cite{Ada07,AGHH88,AK1,ArlTse05,Ash10,BF,bmn,hk,Koc82,LyaMaj, Pos08}).
 Starting  from  fundamental paper \cite{BF} by  Berezin and Faddeev, operators associated  in $L^2(\mathbb{R}^3)$ with the differential expression
 \begin{equation}\label{eq0}
 -\Delta+\sum\limits_{j=1}^m\alpha_j\delta(\cdot - x_j),\quad \alpha_j\in \mathbb{R},\,\,\,m \in\mathbb{N}
 \end{equation}   have  been   treated in  the
 framework  of the  extension  theory. Namely,  the  authors proposed, in the case  of  one  point interaction,  to  consider  all  self-adjoint   extensions  of the following minimal
  Schr\"odinger operator
  \begin{equation}\label{min}
H=-\Delta\upharpoonright\dom(H),\quad \dom(H) :=\bigl\{f \in
W^2_2(\bR^3): f(x_j)=0,\quad j\in \{1,..,m\} \bigr\}
\end{equation}
   as  a realizations of   expression \eqref{eq0}.

 It is well known that    $H$ is  closed nonnegative symmetric
operator with
 deficiency  indices $n_{\pm}(H)=m$ (cf. \cite{AGHH88}). In  \cite{AGHH88}, the  authors proposed to associate   with  Hamiltonian \eqref{eq0} a certain    $m$-parametric  family  $H_{\alpha,X}^{(3)}$ of self-adjoint extensions of the operator $H$. They   parametrized the  extensions  $H_{\alpha,X}^{(3)}$ in terms
 of the  resolvents. The latter  enabled   them to describe  the spectrum
  of  the   $H_{\alpha,X}^{(3)}$.


  In  the recent  publications \cite{bmn,hk},  boundary  triplets and the corresponding Weyl  functions concept
   (see   \cite{DM91, GG} and also  Section \ref{prelim}) was involved
  to investigate  multi-dimensional Schr\"odinger operators with point
  interactions.   In \cite{Ash10,bmn,hk}, two- and three-dimensional Schr\"odinger operators with  one point  interaction
   were  studied.   

   In  the present paper, we  apply  boundary  triplets and  the corresponding  Weyl
   functions approach   to  study  the  matrix   multi-dimensional  Schr\"odinger
   operators with point  interactions. Namely, in $L^2(\bR^d,\bC^n)$ ($d\in\{2,3\}$), we consider
    the  following    matrix  Schr\"{o}dinger  differential expression
 with  singular  potential localized on  the set  $X:=\{x_j\}_{j=1}^m\subset \bR^d$
\begin{equation}\label{eq1}
 -\Delta\otimes I_n+\sum\limits_{j=1}^m\Lambda_j\delta(\cdot - x_j),\quad \Lambda_j\in \mathbb{R}^{n\times n},\,j\in\{1,..,m\}.
 \end{equation}
The minimal  symmetric  operator associated   with  this
expression
 in $L^2(\bR^d,\bC^n)$ is defined  by
\begin{equation}\label{eq2}
H := -\Delta\otimes I_n, \quad \dom(H) :=\bigl\{f \in
W^2_2(\bR^d,\bC^n): f(x_j)=0,\quad x_j\in X \bigr\}.
\end{equation}

The matrix  three-dimensional  Schr\"odinger
   operator  with  one  point  interaction  was  studied in
   \cite{bmn}. 
      We generalize  the results of  \cite{bmn}   to the case of $m$ point
      interactions  and $d=2,3$.
    Namely,  we construct a  boundary triplet $\Pi$  for $H^*$. Moreover, we  compute
    the corresponding  Weyl  function  and
 the  $\gamma$-field for $\Pi$, as well as the  scattering  matrix for  a pair $\{H_0,H_\Theta\}$.
 It is  worth  to mention   that Weyl  function   coincides with matrix-valued  function appearing in
 the formulas  of the resolvents  of $H_{\alpha,X}^{(d)},\,\, d=2,3,\,$ in \cite[chapters II.1, II.4]{AGHH88}.

  In  addition,  we  describe proper,  self-adjoint, and  nonnegative  self-adjoint  extensions of
  the initial  minimal  symmetric  operator  $H$ and  characterize  their
  spectra.     In  particular,  we   show  that the  family $H_{\alpha,X}^{(d)}$  might  be parametrized by means  of diagonal matrices (see  Remark  4.8).  In the case $n=1$,  we   establish  numerous
  links  between our   results and  the results obtained  in the
  previous   publications mentioned above.

 In  Theorem \ref{crit1}  we  establish a  connection
between the   result on
 uniqueness of  nonnegative  self-adjoint  extension of  an arbitrary   nonnegative symmetric operator $A$
   in \cite[Priposition 10]{DM91}  and the  recent   result of V. Adamyan \cite[Theorem
   2.4]{Ada07}. Particularly, we   reproved the  result  on  the uniqueness of  nonnegative  self-adjoint  extension of
   the   minimal  symmetric  operator  $H$ in  the case
   $n=1$ and $d=2$.

 Let   us  briefly  review the structure  of  the paper.
 Section 2  is  introductory. It contains  definitions  and facts  necessary   for  further
 exposition. In Section 3,   we  establish the uniqueness criterion  mentioned above.
 In   Sections 4 and
 5,   we   investigate    the  matrix Schr\"odinger
  operators with  point  interactions   in the  cases   $d=3$ and  $d=2$,
  respectively.
    Namely, in Subsection 4.1(resp., 5.1),  we  define  boundary triplet  for  the $H^*$ and also compute the  corresponding
    Weyl  function and  the $\gamma$-field. The description  of the
    extensions  of  $H$ is  provided   in   Subsection 4.2 (5.2).
    Finally, Subsection 4.3 (5.3) is devoted  to the  spectral
    analysis of the   self-adjoint  extensions of $H$.
\

\

\textbf{Notation.}  Let   $\gotH$ and $\kH$ stand  for separable
Hilbert spaces; $[{\gotH}, {\kH}]$ stands for the space of bounded
linear operators from ${\gotH}$ to ${\kH}$, $[\kH]:=[\kH,\kH]$;
the set of closed operators in $\kH$ is denoted by $\kC(\kH)$. Let
$A$ be a linear operator in a Hilbert space $\mathfrak{H}$. In
what follows,  $R_z(A)$  denotes  the  resolvent $(A-z)^{-1}$ of
the operator
  $A$; $\dom (A)$, $\ker (A)$, $\ran (A)$
 are the domain,  the kernel, and the range of $A$, respectively; $\sigma(A)$ and $\rho (A)$
denote the spectrum and the resolvent set of $A$; $\mathcal{N}_z$
stands for the defect subspace of $A$  corresponding to eigenvalue $z$.
 Denote by $C_0^\infty(\mathbb{R}^d\setminus X)$ the space of infinitely differentiable functions with compact support.


\section{Prelimimaries}\label{prelim}

\subsection{Boundary triplets and  Weyl  functions}
In this subsection, we recall basic notions and facts of the
theory of boundary triplets (we refer the reader  to \cite{DM91, GG} for a
detailed exposition).

\subsubsection{ Linear relations, boundary triplets and  proper
extensions}
 {\bf 1.} The set $\widetilde\cC(\cH)$ of closed linear
relations in $\cH$ is the set of closed linear subspaces of
 $\cH\oplus\cH$. 
  Recall that $\dom(\Theta) =\bigl\{
f:\bigl(\begin{smallmatrix} f \\
f'\end{smallmatrix}\bigr)\in\Theta\bigr\} $, $\ran(\Theta)
=\bigl\{
f^\prime:\bigl(\begin{smallmatrix} f \\
f'\end{smallmatrix}\bigr)\in\Theta\bigr\} $, and $\mul(\Theta)
=\bigl\{
f^\prime:\bigl(\begin{smallmatrix} 0 \\
f'\end{smallmatrix}\bigr)\in\Theta\bigr\} $ are the domain, the
range, and the multivalued part of $\Theta$. A closed linear
operator in $\cH$ is identified with its graph, so that the set
$\cC(\cH)$  of closed linear operators in $\cH$ is viewed as a
subset of $\widetilde\cC(\cH)$. In particular, a linear relation
$\Theta$ is an operator if and only if the multivalued part
$\mul(\Theta)$ is trivial.  We recall that the adjoint relation
$\Theta^*\in\widetilde\cC(\cH)$ of a linear relation $\Theta$ in
$\cH$ is defined by
\begin{equation*}
\Theta^*= \left\{
\begin{pmatrix} k\\k^\prime
\end{pmatrix}: (h^\prime,k)=(h,k^\prime)\,\,\text{for all}\,
\begin{pmatrix} h\\h^\prime\end{pmatrix}
\in\Theta\right\}.
\end{equation*}
 The linear relation $\Theta$ is said to be {\it symmetric}
if $\Theta\subset\Theta^*$ and self-adjoint if $\Theta=\Theta^*$.
The linear relation $\Theta$ is said to be {\it nonnegative} if
$(k',k)\geq  0$ for all $\binom{k}{k'}\in\Theta$. For the
symmetric relation $\Theta\subseteq\Theta^*$ in $\cH$ the
multivalued part $\mul(\Theta)$ is the orthogonal complement of
$\dom(\Theta)$ in $\cH$. Setting $\cH_{\rm
op}:=\overline{\dom(\Theta)}$ and $\cH_\infty=\mul(\Theta)$, one
verifies that $\Theta$ can be written as the direct orthogonal sum
of a self-adjoint operator $\Theta_{\rm op}$ in the subspace
$\cH_{\rm op}$ and a ``pure'' relation
$\Theta_\infty=\bigl\{\bigl(\begin{smallmatrix} 0 \\ f'
\end{smallmatrix}\bigr):f'\in\mul(\Theta)\bigr\}$ in the
subspace $\cH_\infty$.

 Any closed
linear relation admits  the  following representation (see, for
instance, \cite{mm})
\begin{equation}\label{rel}
\Theta=\{ (h,h')^\top\in \kH\oplus\kH:\, Ch-Dh'=0 \}, \qquad
C,D\in[\cH].
 \end{equation}
 Note that representation \eqref{rel} is not unique.

{\bf 2.} Let $A$ be a closed densely defined symmetric operator in
the Hilbert space $\gotH$ with equal deficiency indices
$n_\pm(A)=\dim\ker(A^*\pm i)\leq\infty$.
\begin{definition}[\cite{GG}]\label{bound}%
A triplet $\Pi=\{\kH,\gG_0,\gG_1\}$ is called a {\rm boundary
triplet} for the adjoint operator $A^*$ of $A$ if $\kH$ is an
auxiliary Hilbert space and
$\Gamma_0,\Gamma_1:\  \dom(A^*)\rightarrow\kH$ are linear mappings such that\\
 $(i)$ the  second  Green identity,
\begin{equation*}\label{GI}
(A^*f,g)_\gotH - (f,A^*g)_\gotH = (\gG_1f,\gG_0g)_\kH -
(\gG_0f,\gG_1g)_\kH,
\end{equation*}
holds for all $f,g\in\dom(A^*)$, and\\
$(ii)$ the mapping $\gG:=(\Gamma_0,\Gamma_1)^\top: \dom(A^*)
\rightarrow \kH \oplus\kH$ is surjective.
\end{definition}
Since $n_+(A)=n_-(A)$, a boundary triplet
$\Pi=\{\kH,\gG_0,\gG_1\}$ for $A^*$ exists and is not unique
\cite{GG}. Moreover, $\dim\kH=n_\pm(A) $ and
$\dom(A)=\dom(A^*)\upharpoonright\ker(\Gamma_0)\cap\ker(\Gamma_1)$.

 A closed extension $\widetilde{A}$ of $A$ is called
\emph{proper} if $A\subseteq\widetilde{A}\subseteq A^*$.
 Two proper extensions $\widetilde{A}_1$ and $\widetilde{A}_2$ of $A$ are called \emph{disjoint}
 if $\dom(\widetilde{A}_1)\cap\dom(\widetilde{A}_2)=\dom(A)$ and \emph{transversal} if, in
 addition,
 $\dom(\widetilde{A}_1)\dotplus\dom(\widetilde{A}_2)=\dom(A^*)$ .
  The set of all proper extensions of $A$,
 $\Ext A$, may be described in the following way.
 \begin{proposition}[\cite{DM91,GG}]\label{propo}
 Let $A$ be a densely defined closed symmetric operator in $\gotH$
with equal deficiency indices and let  $\Pi=\{\kH,\gG_0,\gG_1\}$
be  a boundary triplet for  $A^*.$  Then the mapping
\begin{equation}\label{bij}
\Ext_A\ni A_\Theta\rightarrow \Theta:=\Gamma(\dom(\widetilde{A}))=\{(\Gamma_0f,\Gamma_1f)^\top:\,\,f\in \dom(\widetilde{A})\}
\end{equation}
establishes  a bijective correspondence between the set
$\widetilde\kC(\kH)$ and the set of closed proper extensions
$A_\Theta\subseteq A^*$ of $A$. Furthermore,
\begin{equation*}
(A_\Theta)^*=  A_{\Theta^*}
\end{equation*}
holds for any $\Theta\in\widetilde\kC(\kH)$. The extension
$A_\Theta$ in \eqref{bij} is symmetric (self-adjoint) if and only
if $\Theta$ is symmetric (self-adjoint).
\end{proposition}

 Proposition \ref{propo}  and representation \eqref{rel}
 yield  the following corollary.

\begin{corollary}

$(i)$ The extensions $A_0:=A^*\!\upharpoonright\ker(\gG_0)$ and
$A_1:=A^*\!\upharpoonright\ker(\gG_1)$ are self-adjoint.
$(ii)$ Any proper extension $A_\Theta$ of the operator $A$ admits
the representation
\begin{equation}\label{ext repr}
A_\Theta=A_{C,D}=A^*\upharpoonright\dom(A_{C,D}),\quad \dom
(A_{C,D})=\dom
(A^*)\!\upharpoonright\ker(D\Gamma_1-C\Gamma_0),\quad C,D\in[H].
\end{equation}
$(iii)$ If,  in  addition,   the closed extensions $A_\Theta$ and
$A_0$ are disjoint, then   \eqref{ext repr} takes the form
\begin{equation*}\label{bijop}
A_\Theta=A_B=A^*\!\upharpoonright\dom(A_B) ,\quad \dom (A_B)=\dom
(A^*)\!\upharpoonright\ker\bigl(\Gamma_1-B\Gamma_0\bigr),\quad B\in\mathcal{C}(\mathcal{H}).
\end{equation*}
\end{corollary}
\begin{remark}
 In  the  case  $\dim(\cH)<\infty$,   it  follows   from the  result of  Rofe-Beketov  \cite{Rof69}  that
 the  extension $A_\Theta$  defined  by \eqref{ext repr}  is self-adjoint  if and only if  the following
 conditions  hold
\begin{equation}\label{s-acond}
CD^*=DC^*,\qquad 0\in\rho(CC^*+DD^*).
\end{equation}
\end{remark}
\subsubsection{Weyl functions, $\gamma$-fields, and Krein  type formula  for  resolvents}
\begin{definition}[{\cite{DM91}}]\label{Weylfunc}
Let $\Pi=\{\kH,\gG_0,\gG_1\}$ be a boundary triplet  for $A^*.$
The operator valued functions $\gamma(\cdot) :\
\rho(A_0)\rightarrow  [\kH,\gotH]$ and  $M(\cdot) :\
\rho(A_0)\rightarrow  [\kH]$ defined by
\begin{equation}\label{2.3A}
\gamma(z):=\bigl(\Gamma_0\!\upharpoonright\kN_z\bigr)^{-1}
\qquad\text{and}\qquad M(z):=\Gamma_1\gamma(z), \quad
z\in\rho(A_0),
      \end{equation}
are called the {\em $\gamma$-field} and the {\em Weyl function},
respectively, corresponding to the boundary triplet $\Pi.$
\end{definition}
%
The $\gamma$-field $\gamma(\cdot)$
and the Weyl function $M(\cdot)$ in \eqref{2.3A} are well defined.
Moreover, both $\gamma(\cdot)$ and $M(\cdot)$ are
holomorphic on $\rho(A_0)$.
%

 The spectra of the closed (not
necessarily self-adjoint) extensions of $A$ can be described with
the help of the function $M(\cdot)$.
\begin{proposition}\label{prop_II.1.4_spectrum}
Let $\Theta\in \widetilde{\mathcal{C}}(\mathcal{H})$, $A_\Theta\in \Ext_A$, and $z\in \rho(A_0)$.  Then
\[
z\in\sigma_i(A_\Theta) \quad \Leftrightarrow\quad 0\in \sigma_i(\Theta-M(z)),\qquad i\in\{ p,\ c,\ r\}.
\]
\end{proposition}
  Moreover, for
$z\in\rho(A_0)\cap\rho(A_\Theta)$ the  resolvent formula
\begin{equation}\label{2.8}
R_z(A_\Theta) = R_z(A_0) + \gga(z)\bigl(\Theta -
M(z)\bigr)^{-1}\gga(\overline{z})^*,\qquad
z\in\rho(A_\Theta)\cap\rho(A_0)
\end{equation}
holds (see \cite{DM91}). Formula \eqref{2.8} is a generalization of
the  well-known Krein formula for canonical resolvents. We
emphasize that it is valid for any closed extension
$A_\Theta\subseteq A^*$ of $A$ with  nonempty resolvent set.

According to the  representation \eqref{ext repr}, it reads (see
\cite{mm})
\begin{equation}\label{krein}
 R_z(A_\Theta) = R_z(A_0) + \gga(z)\bigl(C -DM(z)\bigr)^{-1}D\gga(\overline{z})^*,\quad
 z\in\rho(A_{C,D})\cap\rho(A_0).
 \end{equation}

Let  now $A$ be a  closed densely  defined   nonnegative symmetric
operator in the  Hilbert  space $\mathfrak H$.  Among  its
nonnegative self-adjoint extensions two extremal extension $A_F$
and $A_K$ are laid special emphasis on. They  are called
Friedrichs and Krein extension, respectively, (see \cite{Kre47}).
Operator $\widetilde{A}$ is nonnegative self-adjoint extension of
$A$ if and only  if $A_K\leq\widetilde{A}\leq A_F$ in  the sense
of the corresponding quadratic  forms.

\begin{proposition}[\cite{DM91, DM95}]\label{prkf}
Let $A$ be  a densely defined nonnegative symmetric operator  with  finite   deficiency indices in
$\gotH$, and  let   $\Pi=\{\kH,\gG_0,\gG_1\}$ be a boundary
 triplet for  $A^*$ such that  $A_0\geq 0$. Let  also  $M(\cdot)$ be the corresponding
Weyl function. Then the  following assertions hold.
\newline
$(i)$   There exists a strong   resolvent  limit
\[
M(0):=s-R-\lim\limits_{x\uparrow 0}M(x),\qquad
(M(-\infty):=s-R-\lim\limits_{x\downarrow -\infty}M(x)).
\]
\newline
 $(ii)$ $M(0)$ ($M(-\infty)$) is  a
 self-adjoint linear relation in $\cH$ associated with  the
 semibounded  below (above) quadratic  form  $\mathrm{t}_0[f]=\lim\limits_{x\uparrow 0}(M(x)f,f)\geq \beta$\quad
  (resp. $\mathrm{t}_{-\infty}[f]=\lim\limits_{x\downarrow
 -\infty}(M(x)f,f)\leq \alpha$)  with  the  domain
 \begin{gather*}\label{Weylform}
 \dom(\mathrm{t}_0)=\{f\in \cH:\,\lim\limits_{x\uparrow 0}|(M(x)f,f)|<\infty\}=\dom((M(0)_{op}-\beta)^{1/2})\\
\dom(\mathrm{t}_{-\infty})=\{f\in \cH:\,\lim\limits_{x\downarrow
 -\infty}|(M(x)f,f)|<\infty\}=\dom((\alpha-M(-\infty)_{op}))^{1/2}).
 \end{gather*}
 Moreover,
 \begin{gather*}
\dom(A_K)=\{f\in\dom(A^*):\,\,(\Gamma_0f,\Gamma_1f)^\top\in M(0)\}\\
(\text{resp. }
\dom(A_F)=\{f\in\dom(A^*):\,\,(\Gamma_0f,\Gamma_1f)^\top\in
M(-\infty)\}).
 \end{gather*}
 $(iii)$ Extensions
$A_0$  and $A_K$  are  disjoint  ($A_0$ and $A_F$ are disjoint) if
and only if  $M(0)\in\cC(\cH)$ ($M(-\infty)\in\cC(\cH)$  resp.)
Moreover,
\begin{equation*}
\dom(A_K)=\dom(A^*)\upharpoonright\ker(\Gamma_1-M(0)\Gamma_0)\qquad
(\dom(A_F)=\dom(A^*)\upharpoonright\ker(\Gamma_1-M(-\infty)\Gamma_0)).
\end{equation*}
\newline
$(iv)$  $A_F=A_0$ ($A_K=A_0$) if  and only  if
 \begin{equation}\label{cond un}
 \lim_{x\downarrow-\infty}(M(x)f,f)=-\infty\qquad(\lim_{x\uparrow0}(M(x)f,f)=+\infty),\quad
 f\in\kH\setminus\{0\}.
 \end{equation}
\newline
$(v)$ If $A_0=A_F$  and    $\dom(\mathrm{t}_{\Theta_{op}})\subset \dom(\mathrm{t}_0)$,  then the   number   of negative eigenvalues   of  self-adjoint extension $A_\Theta$
of $A$ equals   the  number  of   negative  eigenvalues  of  the  quadratic form  $\mathrm{t}_{\Theta_{op}}-\mathrm{t}_0$,  i.e.,
\begin{equation*}
\kappa_-(A_\Theta)=\kappa_-(\mathrm{t}_{\Theta_{op}}-\mathrm{t}_0).
\end{equation*}
Moreover,  if
$M(0)\in[\cH]$,   then  $\kappa_{-}(A_\Theta)=\kappa_{-}(\Theta-M(0))$.

$(vi)$   In  particular,  the  $A_\Theta$   is    nonnegative self-adjoint  if  and  only  if
\begin{equation}\label{nonselfcond}
\dom(\mathrm{t}_{\Theta_{op}})\subset \dom(\mathrm{t}_0)\qquad
\text{and}\qquad \mathrm{t}_{\Theta_{op}}-\mathrm{t}_0\geq 0.
\end{equation}

  If  $M(0)\in[\cH]$,    the inequality  in \eqref{nonselfcond}  takes the
form     $\Theta-M(0)\geq 0$.
\end{proposition}

\subsection{Scattering matrices}

Let $A$ be a densely defined closed symmetric operator in the
separable Hilbert space $\gotH$   with  equal  finite deficiency
 indices and  let  $\Pi =\{\kH,\gG_0,\gG_1\}$ be
a boundary triplet for $A^*$. Assume that $A_\Theta$ is a
self-adjoint extension of $A$ with
$\Theta=\Theta^*\in\widetilde\cC(\cH)$. Since here $\dim\cH$ is
finite,  by \eqref{2.8},
\begin{equation*}
(A_\Theta-z)^{-1}-(A_0-z)^{-1},\qquad
z\in\rho(A_\Theta)\cap\rho(A_0),
\end{equation*}
is a finite rank operator and therefore the pair $\{A_\gT,A_0\}$
performs a so-called {\it complete scattering system}, that is,
the {\it wave operators}
\begin{equation*}
W_\pm(A_\Theta,A_0) :=
\slim_{t\to\pm\infty}e^{itA_\Theta}e^{-itA_0}P^{ac}(A_0),
\end{equation*}
exist and their ranges coincide with the absolutely continuous
subspace $\gotH^{ac}(A_\Theta)$ of $A_\Theta$, cf. \cite{Ka1,Y}.
 $P^{ac}(A_0)$ denotes the orthogonal
projection onto the absolutely continuous subspace
$\gotH^{ac}(A_0)$ of $A_0$. The {\it scattering operator}
$S(A_\Theta,A_0)$ of the {\it scattering system}
$\{A_\Theta,A_0\}$ is then defined by
\begin{equation*}
S(A_\Theta,A_0):= W_+(A_\Theta,A_0)^*W_-(A_\Theta,A_0).
\end{equation*}
If we regard the scattering operator as an operator in
$\gotH^{ac}(A_0)$, then $S(A_\Theta,A_0)$ is unitary, commutes
with the absolutely continuous part
$A^{ac}_0:=A_0\upharpoonright \dom(A_0)\cap\gotH^{ac}(A_0)$
of $A_0$.  It follows that $S(A_\Theta,A_0)$ is unitarily
equivalent to  multiplication operator induced by a family
$\{S_\Theta(z)\}$ of unitary operators in a spectral
representation of $A_0^{ac}$  (for details, see \cite[Section
2.4]{Y}).  Define a  family  of  Hilbert  spaces
$\{\kH_z\}_{z\in\gL^M}$  by
 \begin{displaymath}
\kH_z := \ran\bigl(\imm(M(z + i0))\bigr) \subseteq \kH, \qquad z
\in \gL^M,
\end{displaymath}
where  $M(z + i0) = s-\lim\limits_{\epsilon\to 0}M(z + i\epsilon)$
and  $\gL^M := \bigl\{z \in \bR: M(z + i0) \;
\mbox{exists}\bigr\}.$

In the  following  theorem the  scattering matrix is  calculated
  in the  case of a simple operator $A$.  Recall that
symmetric operator $A$ densely defined in $\gotH$ is said  to be
\emph{simple} if there is no nontrivial subspace which  reduces it
to  a self-adjoint operator.

 \begin{theorem}\cite{bmn}\label{trep}
Let   $A$ be as above, and let  $\Pi=\{\kH,\Gamma_0,\Gamma_1\}$ be   a boundary  triplet  for  $A^*$
   with the
 corresponding  Weyl  function  $M(\cdot)$. Assume also that
 $\Theta=\Theta^*\in\widetilde\kC(\kH)$  and  $A_\Theta$  is  a
 self-adjoint  extension of $A$. Then the  scattering  matrix
$\{S_\Theta(z)\}_{z \in \bR}$ of  the scattering system
$\{A_\Theta,A_0\}$ admits  the  representation
\begin{equation*}\label{scatformula}
S_\Theta(z) = I_{\kH_z} + 2i\sqrt{\imm
(M(z))}\bigl(\Theta-M(z)\bigr)^{-1} \sqrt{\imm (M(z))}\in
[\cH_z],\quad\text{ for a.e. } z \in \gL^M.
\end{equation*}
\end{theorem}
\section{Abstract  description  of  nonnegative self-adjoint
extensions}

 Let $A$  be  a densely defined nonnegative  closed
symmetric operator in $\mathfrak{H}$.   A  complete  description
of  all nonnegative  self-adjoint  extensions of $A$,  as  well as
uniqueness criterion  for   nonnegative  self-adjoint   extension,
has  originally been obtained by Krein in \cite{Kre47} (see also
\cite{ag}). His  results  were  generalized in  numerous works
(see for instance  \cite{Ada07, ArlTse05,DM91} and  reference
therein).  Particularly,  a description in terms of boundary
triplets and
    the  corresponding  Weyl  functions   was  obtained in  \cite[Theorem 4, Proposition 5]{DM91}(cf.  Proposition
    \ref{prkf}  in  Section 2).

One   more uniqueness criterion   has recently been  presented  by
 V. Adamyan \cite[Theorem 2.4]{Ada07}. In this section, we show that
this criterion  might be obtained in the framework of boundary
triplets approach. We  also find  the description   of all
nonnegative self-adjoint extensions of  $A$  similar  to  that of
Adamyan in the particular case $A>\mu I>0.$

\begin{theorem}\label{crit1}
Let $\widetilde{A}_0$ be  a nonnegative self-adjoint extension  of
a nonnegative  closed symmetric  operator $A$   in $\mathfrak{H}$,
 and let  $P_{-1}$ be   an orthogonal projector from $\mathfrak{H}$ onto
$\kN_{-1}$. 
Then $\widetilde{A}_0$ is a unique nonnegative self-adjoint
extension of  $A$ if and only if
\begin{gather}\label{criterion}
\lim\limits_{\varepsilon\downarrow0}(P_{-1}(\widetilde{A}_0+1)(\widetilde{A}_0+\varepsilon)^{-1}\upharpoonright\mathcal{N}_{-1})^{-1}=0,\\
\lim\limits_{\varepsilon\downarrow0}(P_{-1}(\widetilde{A}_0+1)(\varepsilon\widetilde{A}_0+I)^{-1}\upharpoonright\mathcal{N}_{-1})^{-1}=0.\label{criterion'}
\end{gather}

\end{theorem}
 \begin{proof}
  It is well known  (see,  for instance, \cite{DM91}) that  for  each pair of
transversal extensions $\widetilde{A}_1$  and $\widetilde{A}_0$
there exists boundary  triplet $\Pi=\{\mathcal{H},
\Gamma_0,\Gamma_1\}$ such  that
$\ker\Gamma_i=\dom(\widetilde{A}_i),\,i\in\{0,1\}$. In particular,
such boundary  triplet may  be constructed  for the  pair
$\widetilde{A}_0\geq 0$ and $\widetilde{A}_1$, where
$\dom(\widetilde{A}_1)=\dom(A)\dotplus\kN_{-a},\,a>0$. In  this
case  setting
\begin{equation}\label{eq31}
\mathcal{H}=\kN_{-a},\quad
\Gamma_1=P_{-a}(\widetilde{A}_0+a)P_1,\quad \Gamma_0=P_0,
\end{equation}
where $P_{-a}$ is the orthogonal projector from $\mathfrak{H}$
onto $\kN_{-a}$ and $P_1, P_0$ are the projectors from
$\dom(A^*)=\dom(\widetilde{A_0})\dotplus\kN_{-a}$ onto
$\dom(\widetilde{A_0})$ and $\kN_{-a}$, respectively we obtain a
boundary triplet $\Pi=\{\mathcal{H}, \Gamma_0,\Gamma_1\}$ (see
\cite{DM91}). The corresponding Weyl function is
\begin{equation}\label{eq30'}
M_a(z)=(z+a)P_{-a}[I+(z+a)(\widetilde{A_0}-z)^{-1}]=(z+a)P_{-a}(\widetilde{A_0}+I)(\widetilde{A_0}-z)^{-1}.
\end{equation}
Put $a=1$.
 Then  conditions  \eqref{cond un}  take the form
 \begin{gather}\label{crit}
\mathrm{t}_0[f]=\lim\limits_{\varepsilon\downarrow0}\mathrm{t}_0^{\varepsilon}[f]:=\lim\limits_{\varepsilon\downarrow0}((1-\varepsilon)P_{-1}(\widetilde{A}_0+1)(\widetilde{A}_0+\varepsilon)^{-1}f,f)=+\infty\\
\mathrm{t}_{-\infty}[f]=\lim\limits_{\varepsilon\downarrow0}\mathrm{t}_{-\infty}^{\varepsilon}[f]:=
\lim\limits_{\varepsilon\downarrow0}((\varepsilon-1)P_{-1}(\widetilde{A}_0+1)(\varepsilon\widetilde{A}_0+I)^{-1}f,f)=-\infty,\quad\label{crit'}
f\in\kN_{-1}.
\end{gather}
  Since    $\mathrm{t}_0^{\varepsilon}[f]$ is non-decreasing
  semi-bounded from below $(0<\varepsilon<1)$ family of the  closed  symmetric forms,
  \eqref{crit}  is equivalent to \eqref{criterion} (cf. \cite{Ka1}). Analogously,
  since $\mathrm{t}_\infty^{\varepsilon}[f]$  is  non-increasing  semi-bounded from
  above  family of the  closed  symmetric forms, \eqref{crit'} is
  equivalent to \eqref{criterion'}.
  Therefore, by Proposition \ref{prkf}$(iv)$,
the   equality $A_K=A_F$   and, consequently, the
 uniqueness  of nonnegative self-adjoint extension of $A$   is equivalent to the conditions
 \eqref{criterion}-\eqref{criterion'}  (see \cite{Kre47}).
 \end{proof}
Assume now that $A>\mu I>0$  and $\widetilde{A}_0=A_F$ in
\eqref{eq31}. Let   also  $a=1$. According  to Proposition
\ref{prkf}$(vi)$, the following description of  all nonnegative
self-adjoint extensions of  $A$ is  valid.
\begin{proposition}\label{abstr}
Let  $A$    and  $\widetilde{A}_0$ be  as above.  Then the set of
all nonnegative self-adjoint extensions $A_Y$ of $A$ might be
described as follows
\begin{equation*}
\dom(A_Y)=\dom(A^*)\upharpoonright\ker\{Y\Gamma_1-\Gamma_0\},
\end{equation*}
where  $\Gamma_0, \Gamma_1$  are  defined by \eqref{eq31} and  $Y$
runs over the set of all nonnegative contractions in $\kN_{-1}$
satisfying the inequality $0\leq Y\leq M^{-1}_1(0)$ with
$M_1(\cdot)$ defined  by \eqref{eq30'}.
\end{proposition}
\begin{proof}
 It  is easily seen   that $M_1(0)\in[\cH]$  since  $\widetilde{A}_0=A_F>\mu
 I>0$. Thus,  by   Proposition \ref{prkf},   any nonnegative
self-adjoint extension $A_\Theta$ is described by the condition
$\Theta-M_1(0)\geq 0.$ By \eqref{eq30'}, $\Theta\geq M_1(0)\geq 1$.
Therefore $\Theta^{-1}\in\cC(\cH)$ and $0\leq \Theta^{-1} \leq 1$,
i.e., in \eqref{ext repr} $C^{-1}$ exists and
$\Theta^{-1}=C^{-1}D\leq 1$. Putting $Y:=C^{-1}D$, we obtain the
desired  result.
\end{proof}
\section{Three-dimensional Schr\"{o}dinger  operator with  point  interactions}
 Consider in  $L^2(\bR^3,\bC^n)$   matrix  Schr\"{o}dinger  differential
 expression \eqref{eq1}
(see \cite{Ada07,AGHH88,AK1,ArlTse05,BF,bmn, hk, Pos08}). Minimal
symmetric  operator $H$   associated with
 \eqref{eq1} is  defined by \eqref{eq2}.

 Notice  that   $H$ is  closed since for any  $x\in \bR^3$ the
 linear  functional   $\delta_x:\ f\to f(x)$ is continuous  in
$W^2_2(\bR^3,\bC^n)$ due to the Sobolev  embedding theorem.
From the  scalar  case  it is might be easily  derived that  deficiency   indices of $H$ are
$n_{\pm}(H)=mn$.

 \subsection{Boundary  triplet  and Weyl  function}
 In  the following proposition  we  define a boundary  triplet  for the adjoint  $H^*$.
 For  $x=(x^1,x^2,x^3)\in\bR^3$ we agree to write
\begin{equation*}
r_j:=\vert
x-x_j\vert=\sqrt{(x^1-x^1_j)^2+(x^2-x^2_j)^2+(x^3-x^3_j)^2}.
\end{equation*}
\begin{proposition}\label{pr1}
Let  $H$ be the minimal  Schr\"{o}dinger  operator  \eqref{eq2}.
Then  the  following  assertions hold
\newline
 $(i)$ The domain  of  $H^*$  is given  by
\begin{equation}\label{eq3}
\dom(H^*) =\left\{ f=
\sum\limits^m_{j=1}\bigl(\xi_{0j}\frac{e^{-r_j}}{r_j}+\xi_{1j}e^{-r_j}\bigr) +
f_H:\ \xi_{0j},\xi_{1j}\in\bC^n,\, f_H\in\dom(H)\right\}.
\end{equation}
$(ii)$ A boundary  triplet    $\Pi =\{\kH,\Gamma_0,\Gamma_1\}$ for
$H^*$ is defined  by
\begin{gather}
\kH=\oplus_{j=1}^m\bC^n,\quad \Gamma_0
f:=\{\Gamma_{0j}f\}^m_{j=1}= 4\pi\,\{\lim_{x\rightarrow
x_j}f(x)|x-x_j|\}^m_{j=1}=4\pi\{\xi_{0j}\}^m_{j=1},\label{g0}\\
\Gamma_1 f:=\{\Gamma_{1j}f\}^m_{j=1}=\{\lim_{x\rightarrow
x_j}\bigl(f(x)-\tfrac{\xi_{0j}}{|x-x_j|}\bigr)\}^m_{j=1}.\label{g1}
\end{gather}
$(iii)$ The  operator $H_0=H^*\upharpoonright \ker(\Gamma_0)$
 is self-adjoint  with  $\dom(H_0)=W^2_2(\bR^3,\bC^n)$.

\end{proposition}
\begin{proof}
 $(i)$
Without loss of generality, it can be assumed that $n=1$.

 Let  us
show  that the functions  $f_j= e^{-r_j}/r_j$ and $g_j=e^{-r_j}$
($j\in\{1,..,m\}$) belong to $\dom(H^*)$, i.e.,
\begin{equation}\label{domH1}
(H\varphi,e^{-r_j}/r_j)=(\varphi, H^*(e^{-r_j}/r_j)) \text{ and }
(H\varphi,e^{-r_j})=(\varphi, H^*(e^{-r_j})),\quad  \varphi\in
C_0^\infty(\mathbb{R}^3\setminus X).
\end{equation}
Let $u(\cdot),v(\cdot)\in C^2(\Omega)\cap C^1(\overline{\Omega})$.
Then the second Green formula reads  as follows
\begin{equation}\label{green}
\int\limits_\Omega\biggl(\Delta
u(\mathrm{x})\overline{v(\mathrm{x})}
-u(\mathrm{x})\overline{\Delta
v(\mathrm{x})}\biggr)\mathrm{dx}=\int\limits_{\partial\Omega}\biggl(\frac{\partial
u(\mathrm{s})}{\partial
n}\overline{v(\mathrm{s})}-u(\mathrm{s})\overline{\frac{\partial
v(\mathrm{s})}{\partial n}}\biggr)\mathrm{ds}.
\end{equation}
 By  \eqref{green}, we obtain
\begin{multline}\label{domH2}
(H\varphi,e^{-r_j}/r_j)-(\varphi,
H^*(e^{-r_j}/r_j))=\lim\limits_{r\rightarrow\infty}\int\limits_{B_r(x_j)\backslash
B_{\tfrac1{r}}(x_j)}\biggl(-\Delta
\varphi\frac{e^{-r_j}}{r_j}+\varphi\Delta(\frac{e^{-r_j}}{r_j})\biggr)\mathrm{dx}\\
=\lim\limits_{r\rightarrow\infty}\int\limits_{S_r(x_j)}\biggl(-\frac{\partial{\varphi}}{\partial
n}\frac{e^{-r_j}}{r_j}+\varphi\frac{\partial}{\partial
n}\left(\frac{e^{-r_j}}{r_j}\right)\biggr)\mathrm{ds}+\lim\limits_{r\rightarrow\infty}\int\limits_{S_{\tfrac1{r}}(x_j)}\biggl(\frac{\partial{\varphi}}{\partial
n}\frac{e^{-r_j}}{r_j}-\varphi\frac{\partial}{\partial
n}\left(\frac{e^{-r_j}}{r_j}\right)\biggr)\mathrm{ds}.
\end{multline}
It  is easily seen that $\frac{\partial}{\partial
n}\left(\frac{e^{-r_j}}{r_j}\right)=-\tfrac{e^{-r_j}}{r_j}(1+\tfrac1{r_j})$.
Therefore the first integral  in the right-hand side of
\eqref{domH2} tends to 0 as $r\rightarrow \infty$ since
$\varphi\in C_0^\infty(\mathbb{R}^3\setminus X).$ Further,
\begin{gather*}
\lim\limits_{r\rightarrow\infty}\int\limits_{S_{\tfrac1{r}}(x_j)}\frac{\partial{\varphi}}{\partial
n}\frac{e^{-r_j}}{r_j}\mathrm{ds}=\lim\limits_{r\rightarrow\infty}4\pi\frac{\partial{\varphi}}{\partial
n}(x^*)\frac{e^{-1/r}}{r}=0,\quad x^*\in S_{\tfrac1{r}}(x_j),\\
-\frac1{4\pi}\lim\limits_{r\rightarrow\infty}\int\limits_{S_{\tfrac1{r}}(x_j)}\varphi\frac{\partial}{\partial
n}\left(\frac{e^{-r_j}}{r_j}\right)\mathrm{ds}=\lim\limits_{r\rightarrow\infty}\left[\tfrac{e^{-1/r}}{r}(1+r)\varphi(x')\right]=\lim\limits_{x'\rightarrow
x_j}\varphi(x')=\varphi(x_j)=0, \,\, x'\in S_{\tfrac1{r}}(x_j).
\end{gather*}%
Thus, the  first equality of \eqref{domH1} holds. The second one
can   be  proved  analogously.
It is  not  difficult to  show  that  the  functions  $f_j$   and  $g_j$ are  linearly independent  and $\dim(\Span\{f_j,g_j\}_{j=1}^m)=2mn$. Since $\Span\{f_j,g_j\}_{j=1}^m\cap \dom(H)=0$  and
$\dim(\dom(H^*)/\dom(H)) =2mn$, the  domain  $\dom(H^*)$ takes the form
\eqref{eq3}.

$(ii)$  Let  $f,g\in\dom(H^*)$. By  \eqref{eq3}, we  have
 \bed
  f= \sum\limits^m_{k=1}f_k + f_H, \,\,\, f_k =\xi_{0k}\,\frac{e^{-r_k}}{
r_k}+\xi_{1k}e^{-r_k}, \quad
 \quad g = \sum\limits^m_{k=1}g_k + g_H, \,\,\, g_k =
\eta_{0k}\,\frac{e^{-r_k}}{r_k}+\eta_{1k}\,e^{-r_k},
 \eed where $f_H,
g_H\in\dom(H)$,   and  $\xi_{0k},\ \xi_{1k}, \eta_{0k},\
\eta_{1k}\in\bC^n,\ k\in\{1,..,m\}$.

 Applying  \eqref{g0}-\eqref{g1}  to $f,g\in\dom(H^*)$, we obtain
\begin{multline}\label{tripl}
\Gamma_0f=4\pi\{\xi_{0j}\}_{j=1}^m,\quad\Gamma_1f=\left\{-\xi_{0j}+\sum\limits_{k\neq
j}\xi_{0k}\frac{e^{-|x_j-x_k|}}{|x_j-x_k|}+\sum\limits_{k=1}^m\xi_{1k}e^{-|x_j-x_k|}\right\}_{j=1}^m,\\
\Gamma_0g=4\pi\{\eta_{0j}\}_{j=1}^m,\quad\Gamma_1g=\left\{-\eta_{0j}+\sum\limits_{k\neq
j}\eta_{0k}\frac{e^{-|x_j-x_k|}}{|x_j-x_k|}+\sum\limits_{k=1}^m\eta_{1k}e^{-|x_j-x_k|}\right\}_{j=1}^m.\quad\quad\qquad\quad
\end{multline}
It is easily seen that
\begin{multline*}
(H^*f,g) -(f,H^*g) =
\sum\limits^m_{j=1}\sum\limits^m_{k=1}\biggl((\xi_{0j}H^*(\frac{e^{-r_j}}{
r_j}),\eta_{1k}\,e^{-r_k}) - (\xi_{0j}\,\frac{e^{-r_j}}{
r_j},\eta_{1k}H^*(e^{-r_k}))\\+(\xi_{1j}H^*(e^{-r_j}),\eta_{0k}\,\frac{e^{-r_k}}{r_k})
-(\xi_{1j}\,e^{-r_j},\eta_{0k}H^*(\frac{e^{-r_k}}{r_k}))\biggr).
\end{multline*}
Using the  second Green  formula \eqref{green}, we get
\begin{multline}\label{green'}
(H^*(\frac{e^{-r_j}}{r_j}),e^{-r_k})-
(\frac{e^{-r_j}}{r_j},H^*(e^{-r_k}))=\lim\limits_{r\rightarrow\infty}\biggl(\int\limits_{B_r(x_j)\backslash
B_{\tfrac1{r}}(x_j)}-\Delta(\frac{e^{-r_j}}{r_j})e^{-r_k}\mathrm{dx}\\+\int\limits_{
B_r(x_j)\backslash
B_{\tfrac1{r}}(x_j)}\frac{e^{-r_j}}{r_j}\Delta(e^{-r_k})\mathrm{dx}\biggr)= -4\pi e^{-|x_k-x_j|}.
\end{multline}
 Finally, by \eqref{tripl}  and \eqref{green'},
\begin{multline*}
(H^*f,g)
-(f,H^*g)=4\pi\sum\limits^m_{j=1}\sum\limits_{k=1}^m\biggl(-\xi_{0j}
\overline{\eta}_{1k}e^{-|x_j-x_k|}+
\xi_{1j}\overline{\eta}_{0k}e^{-|x_j-x_k|}\biggr)\\=
\sum\limits_{j=1}^m(\Gamma_{1j}f,\Gamma_{0j}g)-(\Gamma_{0j}f,\Gamma_{1j}g)=(\Gamma_1f,\Gamma_0g)-(\Gamma_0f,\Gamma_1g).
\end{multline*}
 Thus, the Green identity  is satisfied.  It  follows from
 \eqref{eq3} that  the mapping
 $\Gamma=(\Gamma_0,\Gamma_1)^\top$ is surjective. Namely, let $(h_0,h_1)^\top\in\cH\oplus\cH$,
 where $h_0=\{h_{0j}\}_{j=1}^m, h_1=\{h_{1j}\}_{j=1}^m$
 are  vectors from $\oplus_{j=1}^m\mathbb{C}^n$.  If $f\in\dom(H^*)$,
  then, by \eqref{eq3},  $f=f_H+\sum\limits^m_{j=1}\bigl(\xi_{0j}\frac{e^{-r_j}}{r_j}+\xi_{1j}e^{-r_j}\bigr)$.
  Let us  put
\begin{equation}\label{hat}
\xi_0:=\{\xi_{0j}\}_{j=1}^m,\,\,\xi_1:=\{\xi_{1j}\}_{j=1}^m,\,\,
E_0:=\left(-\frac{e^{-|x_k-x_j|}}{|x_k-x_j|-\delta_{kj}}\right)_{j,k=1}^m,\quad
E_1:=\left(e^{-|x_k-x_j|}\right)_{k,j=1}^m,
\end{equation}
where  $\delta_{kj}$  stands  for the Kronecker  symbol.
 Therefore if $\xi_0=\tfrac1{4\pi}h_0$ and
$\xi_1={(E_1\otimes
I_n)}^{-1}(h_1+\tfrac1{4\pi}(E_0\otimes I_n)h_0)$,
then $\Gamma_0f=h_0$  and $\Gamma_1f=h_1$. Hence
assertion $(ii)$  is proved.

$(iii)$ Combining  \eqref{eq2} with   \eqref{eq3},  we   obtain
that
 any   $f\in W^2_2(\bR^3,\bC^n)$  admits the representation
  $f=\sum\limits^m_{j=1}\xi_{1j}e^{-r_j}\, + f_H$ with
 $\sum\limits_{k=1}^m \xi_{k1}e^{-|x_k-x_j|}= f(x_j)$
which proves $ (iii)$.
\end{proof}

In  what  follows
$\sqrt{\cdot}$ stands for the branch of the corresponding
multifunction defined on $\mathbb{C}\setminus \mathbb{R}_+$
by the condition  $\sqrt{1}=1.$
\begin{proposition}\label{pr2}
Let   $H$  be the minimal  Schr\"{o}dinger  operator defined by \eqref{eq2}
and let  $\Pi=\{\kH,\Gamma_0,\Gamma_1\}$ be  the  boundary triplet
for $H^*$ defined by \eqref{g0}-\eqref{g1}. Then

  $(i)$ the  Weyl function   $M(\cdot)$ corresponding  to
 $\Pi$ has the  form
 \begin{equation}\label{W3}
 M(z)=\bigoplus\limits_{s=1}^nM_s(z), \quad
 M_s(z)=\left(\tfrac{i\sqrt{z}}{4\pi}\delta_{jk}+\widetilde{G}_{\sqrt{z}}(x_j-x_{k})\right)_{j,k=1}^m,\,\,\,
 z\in\mathbb{C}_+,
  \end{equation}
  where  $\widetilde{G}_{\sqrt{z}}(x)=\left\{%
  \begin{array}{ll}
    \frac{e^{i\sqrt{z}|x|}}{4\pi|x|}, & \hbox{x$\neq$0;} \\
    0, & \hbox{x=0.} \\
\end{array}%
\right.$ and $\delta_{kj}$ stands  for the  Kronecker symbol;

 $(ii)$  the corresponding $\gamma(\cdot)$-field  is
\begin{equation}\label{g3}
\gamma(z)\overline{\xi}=\sum\limits^m_{j=1}\xi_j\,\frac{e^{i\sqrt{z}r_j}}{4\pi
r_j},\quad \overline{\xi}=\{\xi_j\}_{j=1}^m,\quad
\xi_j\in\bC^n,\quad z\in\bC_+.
\end{equation}
 \end{proposition}
\begin{proof}
 Let  $f_z\in \kN_z,\ z\in\bC_+$. Then $f_z=
\sum\limits^m_{j=1}a_j\,\frac{e^{i\sqrt{z}r_j}}{4\pi r_j},\quad
a_j\in\bC^n$  (see  \cite[chapter II.1]{AGHH88}).

 Applying   $\Gamma_0$  and  $\Gamma_1$ to  $f_z$, we  get
 \begin{equation}\label{eq*}
\Gamma_0f_z=\{a_j\}^m_{j=1},\qquad
 \Gamma_1f_z=\left\{a_j\frac{i\sqrt{z}}{4\pi}+\sum\limits_{k\neq
j}a_k\frac{e^{i\sqrt{z}|x_j-x_k|}}{4\pi|x_j-x_k|}\right\}^m_{j=1}.
\end{equation}
Therefore  \eqref{W3}  is  proved  (see Definition
\ref{Weylfunc}).
 Finally, combining \eqref{eq*}  with \eqref{2.3A}, we  arrive at \eqref{g3}.
 \end{proof}

\begin{remark}
 $(i)$  The  first   construction  of the   boundary    triplet, in the  case  $m=n=1$,   apparently   goes back  to  the  paper   by  Lyantse  and  Majorga
   \cite[Theorem 2.1]{LyaMaj}.  They   also obtained   the  description
of   the spectrum  of an   arbitrary proper  extension
$H_\Theta$ of  $H$  \cite[Theorem 4.5]{LyaMaj}.   Their  description of $(H_\Theta-z)^{-1}$ coincides
with the Krein  formula  for canonical  resolvents in Theorem
\ref{pr1'}.
    Another construction  of the    boundary   triplet  in  the   situation   of  general   elliptic  operator  with the  boundary   conditions  on   the  set  of zero Lebesgue measure   was   obtained  in  \cite{Koc82}.  However   this  construction is  not  suitable  for our purpose.
    In the   case   $m=1$,  slightly  different   boundary triplet was obtained in  \cite[section 5.4]{bmn}.

    $(ii)$  Note  also  that  the  Weyl  function in the  form  \eqref{W3} appears   in  the  paper  by A. Posilicano \cite[Example 5.3]{Pos08} and  in the book \cite{AGHH88} (see  Theorem 1.1.1  in chapter  II.1)  without  connection with boundary  triplets.

    \end{remark}
 \subsection{Proper  extensions of the  minimal Schr\"{o}dinger  operator $H$}
Proposition  \ref{propo} gives a description  of all proper
extensions of  $H$ in  terms of  boundary  triplets. The following
theorem is its reformulation in  more precise  form.
\begin{theorem}\label{pr1'}
Let   $H$  be the minimal  Schr\"{o}dinger  operator \eqref{eq2},
 let  $\Pi=\{\kH,\Gamma_0,\Gamma_1\}$  be the boundary triplet
for  $H^*$ defined  by  \eqref{g0}-\eqref{g1}, and   $M(\cdot)$
  the corresponding Weyl  function. Assume that
$\xi_0,\xi_1, E_0, E_1$ are  defined  by
\eqref{hat} and    $H_{C,D}$  is  a proper extension of
$H$. 
Then  the following assertions hold.
\newline
$(i)$ The set of all  proper  extensions  $H_{C,D}$ of $H$ is
described as follows
\begin{equation}\label{concr}
\dom(H_{C,D})=\left\{f\in \dom(H^*):\,\, D(E_1\otimes
I_n)\xi_1=(4\pi C+D(E_0\otimes
I_n))\xi_0\right\}, \quad C,D\in[\cH].
\end{equation}
$(ii)$  Moreover,   $H_{C,D}$  is a   self-adjoint extension of
$H$ if and only if \eqref{s-acond} holds.
\newline
$(iii)$
 Friedrichs  extension   $H_F$  of  $H$ coincides
 with $H_0$:
\begin{equation*}
\dom(H_F)=\dom(H_0)=W^2_2(\bR^3,\bC^n).
\end{equation*}
\newline
$(iv)$ The  domain  of  Krein  extension   $H_K$ is
\begin{equation}\label{eqHk}
\dom(H_K)=\left\{ f=
\sum\limits^m_{j=1}\xi_{0j}\frac{e^{-r_j}}{r_j}+\sum\limits^m_{k,j=1}\textbf{k}_{jk}\xi_{0k}e^{-r_j}
+ f_H: \, \xi_{0j}\in\bC^n,\, f_H\in\dom(H) \right\},
\end{equation}
with
\begin{gather*}\label{eqHk'}
\textbf{K}=(\textbf{k}_{kj})_{k,j=1}^m=(E_1\otimes I_n)^{-1}(4\pi
M(0)+E_0\otimes I_n),\\
M(0)=I_n\otimes\left(\widetilde{G}_0(x_j-x_{k})\right)_{j,k=1}^m=I_n\otimes
\left(\frac{1-\delta_{jk}}{4\pi|x_k-x_j|+\delta_{jk}}\right)_{j,k=1}^m.\label{Weyl0}
\end{gather*}
$(v)$  Proper  extension   $H_{C,D}$ of the  form \eqref{concr} is
self-adjoint  and  nonnegative  if and  only  if \eqref{s-acond}
holds and
 \begin{equation*}\label{cond pos}
 ((CD^*-DM(0)D^*)h,h)\geq 0,\qquad
h\in\cH\setminus\{0\}.
 \end{equation*}
 $(vi)$ Krein  formula for  canonical  resolvents  takes
the form
 \begin{equation}\label{Kre3}
 R_z(H_{C,D}) = R_z(H_0) + \gga(z)\bigl(C
 -DM(z)\bigr)^{-1}D\gga(\overline{z})^*,\,\,
 z\in\rho(H_{C,D})\setminus\bR_+,
 \end{equation}
where $\gamma(\cdot)$-field is  defined  by \eqref{g3} and
$R_z(H_0)$ is an integral operator  with the kernel
$G_{\sqrt{z}}(x,x')=\frac{e^{i\sqrt{z}|x-x'|}}{4\pi|x-x'|}\otimes
I_n$.
\end{theorem}
\begin{proof}
$(i)$ and $(ii)$  follow from the representation \eqref{ext repr}
and  \eqref{tripl}.

 $(iii)$  It is easily  seen that   \eqref{W3}  implies   $s-R-\lim\limits_{x\downarrow-\infty}M(x)=-\infty I_{nm}$.
 Then, by  Proposition \ref{prkf} $(iv)$,  $H_F=H_0$. Finally, by  Proposition \ref{pr1}$(iii)$,
   $\dom(H_F)=\dom(H_0)=W^2_2(\bR^3,\bC^n)$ .

 $(iv)$ Note that strong  resolvent  limit
$s-R-\lim\limits_{x\uparrow0}M(x)=M(0)=I_n\otimes\left(\widetilde{G}_0(x_j-x_{k})\right)_{j,k=1}^m
=I_n\otimes\left(\frac{1-\delta_{jk}}{4\pi|x_k-x_j|+\delta_{jk}}\right)_{j,k=1}^m$
is  an operator. Therefore operators $H_0$ and  $H_K$   are
disjoint and, by Proposition \ref{prkf}$(iii)$, formula
\eqref{eqHk} is valid.

 $(v)$  follows from  Proposition  \ref{prkf}$(vi)$.

   Finally,  \eqref{krein} and  formula for the   kernel  of $(H_0 - z)^{-1}$ (see \cite[chapter I.1]{AGHH88})  yield  $(vi)$.
\end{proof}

In \cite{AGHH88}, it is  noted that, in  the case $n=1$, according
to  the extension theory, there are $m^2$-parametric  family of
self-adjoint  extensions of  the minimal operator $H$  defined by
\eqref{min}. However, in \cite{AGHH88}, only certain
$m$-parametric family $H_{\alpha,X}^{(3)}$ associated  with the
differential expression \eqref{eq0} is described  \cite[chapter
II.1, Theorem 1.1.3]{AGHH88}.  The    family $H_{\alpha,X}^{(3)}$
might  be parametrized
   in  the framework  of  boundary  triplet approach.
\begin{proposition}
 Let 
   $\Pi$ be  the  boundary triplet for $H^*$ defined  by \eqref{g0}-\eqref{g1}.
   Then the domain  of the Schr\"{o}dinger
 operator   $H_{\alpha,X}^{(3)}$  is
  \begin{equation}\label{diag}
\dom(H^{(3)}_{\alpha,X})=\dom(H^*)\upharpoonright\ker(\Gamma_1-B_\alpha\Gamma_0),
\quad B_\alpha=\diag(\alpha_1,.., \alpha_m),\,\,
\,\,\,\alpha_k\in\mathbb{R},\,\,\, k\in\{1,..,m\}.
 \end{equation}
 \end{proposition}
Note also  that the  description of the  $H_{\alpha,X}^{(3)}$
   in  terms of the resolvents  \cite[chapter II.1]{AGHH88} coincides  with the  Krein
 formula for canonical  resolvents \eqref{Kre3} with $C=B_\alpha=\diag(\alpha_1,.., \alpha_m)$ and $D=I_m$.

\begin{remark}
In  the  case $n=m=1$, formulas \eqref{concr} and \eqref{eqHk} are
essentially  simplified. Namely,
\begin{equation*}
\dom( H_{C,D})=\left\{
f=\xi_{0}\frac{e^{-r_1}}{r_1}+\xi_{1}e^{-r_1} + f_H:\,
d\xi_1=(4\pi c+d)\xi_0,\,\,\xi_{0},\xi_{1}, c, d\in \bC,\quad
f_H\in\dom(H) \right\}
\end{equation*}
and
\begin{equation*}
\dom( H_K)=\left\{ f=\xi_{0}\frac{e^{-r_1}}{r_1}+\xi_{0}e^{-r_1} +
f_H,\,\,\xi_0\in \bC,\quad f_H\in\dom(H) \right\}.
\end{equation*}
 \end{remark}
 \begin{remark}
The  matrix  Schr\"odinder operator   with
 finite  number  of point  interactions  was  also   studied  by
 A. Posilicano \cite[Example 5.3, Example 5.4]{Pos08}.   Particularly, the
 author parametrized   self-adjoint  extensions of  the minimal  symmetric operator $H$.
 A connection  between  our  description  of  self-adjoint  extensions and  the  one
obtained  by  A. Posilicano  might  be  established  by   the
formulas
  (4.5) and  (4.6) in \cite[Theorem 4.5]{Pos08}.
 \end{remark}
\begin{remark}
In \cite{ArlTse05},  Yu.Arlinskii  and  E.Tsekanovskii described
all nonnegative self-adjoint extensions $\widetilde{H}$ of $H$  in
the
case $n=1$ (see \cite[Theorem 5.1]{ArlTse05}). 
It  should  be noted  that  the   description  of all  nonnegative
self-adjoint extensions of $H$ close to that   contained in
\cite{ArlTse05} might be  obtained  in the framework of our scheme.
It will be published elsewhere.
\end{remark}

\subsection{Spectrum  of the  self-adjoint  extensions of  the  minimal Schr\"{o}dinger  operator  and scattering  matrix }

In  this subsection  we describe point  spectrum of the self-adjoint
extensions of $H$  and   complete  some  results  from
\cite{AGHH88} in this  direction.
\begin{theorem}\label{spec3}
Let  $H$ be  the minimal  Schr\"{o}dinger  operator \eqref{eq2},
let  $\Pi$ be  the boundary  triplet  for  $H^*$ defined by
\eqref{g0}-\eqref{g1}, and  $M(\cdot)$    the corresponding
Weyl function defined by  \eqref{W3}. Assume that $H_\Theta$  is a
self-adjoint extension of $H$.  Then the following assertions
hold.
\newline
  $(i)$  Point spectrum of the self-adjoint extension  $H_\Theta$ of $H$ consists  of at  most   $nm$
 negative  eigenvalues (counting multiplicities). Moreover,
$z\in\sigma_p(H_\Theta)\cap\bR_-\,$ if and only if
$0\in\sigma_p(\Theta-M(z))$, i.e., the following equivalence holds
\begin{equation*}\label{p_spec}
z\in\sigma_p(H_\Theta)\cap\bR_- \Leftrightarrow\,0\in\sigma_p(C-
DM(z)).
\end{equation*}
The corresponding   eigenfunction $\psi_z$ has the  form
\begin{equation*}\label{eigen}
\psi_z=\sum\limits_{j=1}^mc_j\frac{e^{i\sqrt{z}r_j}}{4\pi  r_j},
\end{equation*}
where $(c_1,.., c_m)^\top$ is eigenvector  of  the   relation
$\Theta-M(z)$ corresponding  to zero eigenvalue.
\newline
 $(ii)$ The number of  negative eigenvalues of
 the  self-adjoint  extension $H_\Theta$ is equal to the number  of
 negative eigenvalues of  the  relation
$\Theta-M(0)$, $\kappa_{-}(H_\Theta)= \kappa_{-}(\Theta-M(0))$,
i.e., 
\begin{equation*} \kappa_{-}(H_{C,D})=\kappa_{-}(CD^*- D
M(0)D^*), \end{equation*}
 where $M(0)$ is  defined   by
\eqref{Weyl0}.
\end{theorem}
Next   we  find   sufficient
 conditions  for the inequality $\kappa_-(H_{\alpha,X}^{(3)})\geq m'$ (with  $m'\leq m$)
 as well  as  for the equality $\kappa_-(H_{\alpha,X}^{(3)}) = m'$ to hold  by  applying  the  following  Gerschgorin theorem.
\begin{theorem}\cite[Theorem 7.2.1]{Lan}
All eigenvalues of a matrix
$A=(a_{ij})_{i,j=1}^m\in[\mathbb{C}^m]$ are contained in the union
of Gerschgorin's disks
 \begin{equation*}\label{gercond}
G_k=\{z\in\mathbb{C}:\,|z-a_{kk}|\leq\sum\limits_{k\neq
j}|a_{kj}|\},\qquad k\in\{1,..,m\}.
\end{equation*}
Moreover, the set consisting of  $m'$ disks that do not intersect
with remaining $m-m'$ disks contains precisely $m'$ eigenvalues of
the matrix $A$.
\end{theorem}
\begin{proposition}
 Let  $H_{\alpha,X}^{(3)}$  be defined  by \eqref{diag}.  Let
 also $K=\{k_i\}_{i=1}^{m'}$ be   a subset  of \,\, $\mathbb{N}$.

 $(i)$   Suppose  that
 \begin{equation}\label{suff}
 \alpha_{k_i}<-\sum\limits_{j\neq
k_i}\frac 1{4\pi|x_j-x_{k_i}|}\quad \text{ for }\quad k_i\in K.
 \end{equation}
 Then  $\kappa_-(H_{\alpha,X}^{(3)})\geq m'$.

 $(ii)$  If, in addition, $\alpha_k\geq \sum\limits_{j\neq
k}\frac 1{4\pi|x_j-x_k|}$ for  $k\notin K$, then
$\kappa_-(H_{\alpha,X}^{(3)})=m'$.
 \end{proposition}
 \begin{proof}

 $(i)$
 Combining   Theorem  \ref{spec3}$(ii)$     with  \eqref{Weyl0}, we   get
 \begin{equation*}
 \kappa_-(H_{\alpha,X}^{(3)})=\kappa_-(B_\alpha-M(0))=\kappa_-\left(\left(\alpha_k\delta_{jk}-\frac{1-\delta_{jk}}{4\pi|x_j-x_k|+\delta_{jk}}\right)_{j,k=1}^m\right).
 \end{equation*}
 Without   loss of
generality   we  may assume that  $K=\{1,..,m'\}$. Denote   by
$B_{m'}$ the upper left $m'\times m'$   corner   of the matrix
$B_\alpha-M(0)$. According to the  minimax  principle,
\begin{equation}\label{minimax}
 \kappa_-(H_{\alpha,X}^{(3)})=\kappa_-(B_\alpha-M(0))\geq\kappa_-(B_{m'}).
\end{equation}
Conditions  \eqref{suff} yield the   corresponding  Gerschgorin
conditions  for $B_{m'}$. Applying the Gerschgorin theorem to the
matrix $B_{m'}$ and using \eqref{suff}, we get
$\kappa_-(B_{m'})=m'$. Combining  the latter  equation  with
\eqref{minimax}, we  get $\kappa_-(H_{\alpha,X}^{(3)})\geq m'$.

$(ii)$ Applying the second part of the Gerschgorin theorem to the
matrix $B-M(0)$, we arrive  at
$\kappa_-(H_{\alpha,X}^{(3)})=\kappa_-(B_\alpha-M(0))=m'$.
\end{proof}
\begin{remark}
 Note  that the  idea of  applying Gerschgorin's theorem is  borrowed  from \cite{Ogu08}. This  idea was also  used  in  \cite{GolOrid09}.
   \end{remark}

 Consider  the scattering  system
$\{H_\gT,H_0\}$, where
$H_\Theta=H^*\upharpoonright\Gamma^{-1}\Theta$ with  arbitrary
self-adjoint  relation  $\Theta\in\widetilde\cC(\kH)$. Since  $H$
is not simple, we consider the system
$\{\widehat{H}_\gT,\widehat{H}_0\}$, $H_\Theta=\widehat
H_\Theta\oplus H_s$. Then Theorem  ~\ref{trep} and  \eqref{W3}
imply
\begin{theorem}
Scattering matrix  $\{\widehat{S}_\gT(z)\}_{z \in \bR_+}$ of the
scattering system $\{\widehat{H}_\gT,\widehat{H}_0\}$ has the
 form

\begin{gather*} \widehat{S}_\gT(x) = I_{nm} +
2i\sqrt{S(x)}\bigl(\gT -
I_{n}\otimes\left(\tfrac{i\sqrt{x}}{4\pi}\delta_{jk}+
\widetilde{G}_{\sqrt{x}}(x_j-x_{k})\right)_{j,k=1}^m\bigr)^{-1}\sqrt{S(x)},\quad x \in \bR_+,\\
S(x)=I_{n}\otimes\left(\tfrac{\sqrt{x}}{4\pi}\delta_{jk}+
\widetilde{S}_{\sqrt{x}}(x_j-x_{k})\right)_{j,k=1}^m,\quad\widetilde{S}_{\sqrt{x}}(t)=\left\{%
  \begin{array}{ll}
    \frac{\sin(\sqrt{x}|t|)}{4\pi|t|}, & \hbox{t$\neq$0;} \\
    0, & \hbox{t=0.} \\
\end{array}%
\right.
\end{gather*}
\end{theorem}

\section{Two-dimensional  Schr\"{o}dinger  operator  with  point  interactions}
  In this section, we consider  in $L^2(\bR^2,\bC^n)$ matrix Schr\" {o}dinger differential  expression  \eqref{eq1}
(see \cite{Ada07,AGHH88,AK1,hk}).
  Minimal symmetric operator $H$ associated with \eqref{eq1} in  $L^2(\bR^2,\bC^n)$ is defined
  by \eqref{eq2}.  As above, the  operator $H$ is   closed and the deficiency  indices  of  $H$ are  $n_{\pm}(H)=nm.$
%
%
%
 \subsection{Boundary  triplet  and Weyl  function}
 In the following proposition we describe boundary triplet  for the  adjoint operator  $H^*$.
 Let us denote
\begin{equation*}
r_j:=\vert x-x_j\vert=\sqrt{(x^1-x^1_j)^2+(x^2-x^2_j)^2},\quad
x=(x^1,x^2)\in\bR^2.
\end{equation*}

\begin{proposition}\label{pr1a}
Let  $H$ be the minimal Schr\"{o}dinger  operator  defined by \eqref{eq2}.
Then  the following assertions hold.
\newline
 $(i)$ The domain  of   $H^*$ is  defined  by

\begin{equation}\label{eq3a}
\dom(H^*) =\left\{ f= \sum\limits^m_{j=1}\bigl(\xi_{0j}\,e^{-r_j}
\ln(r_j)+\xi_{1j}\,e^{-r_j}\bigr) + f_H\ :\
\xi_{0j},\xi_{1j}\in\bC^n,\quad f_H\in\dom(H)\right\}.
\end{equation}
 $(ii)$
 The boundary  triplet   $\Pi
=\{\kH,\Gamma_0,\Gamma_1\}$  for  $H^*$ might  be defined as
follows
\begin{gather}  \label{T2}
\kH=\oplus_{j=1}^m\bC^n,\quad \Gamma_0
f:=\{\Gamma_{0j}f\}_{j=1}^m=- 2\pi\,\{\lim_{x\rightarrow x_j}\frac
{f(x)}{ \ \ln|x-x_j|}\}^m_{j=1}=2\pi\{\xi_{0j}\}_{j=1}^m,\\
\Gamma_1 f:=\{\Gamma_{1j}f\}_{j=1}^m =\{\lim_{x\rightarrow
x_j}\left(f(x)- \ln|x-x_j|\xi_{0j}\right)\}^m_{j=1},\quad f \in
\dom(H^*)\label{T2'}.
\end{gather}
$(iii)$ The operator    $H_0=H^*\upharpoonright \ker(\Gamma_0)$ is
self-adjoint  with $\dom(H_0)=W^2_2(\bR^2,\bC^n)$.
\end{proposition}
\begin{proof}
$(i)$ It is well known (see \cite{AGHH88,AK1}) that
\begin{displaymath}
\dom (H^*)=\bigl\{f\in L^2(\bR^2,\bC^n)\cap W^2_{2,\text{\rm
loc}}(\bR^2\backslash\{X\},\bC^n): \Delta f\in
L^2(\bR^2,\bC^n)\bigr\}.
\end{displaymath}
  Obviously,   functions $f_j= \eta_{j}e^{-r_j} \ln(r_j)$
 and $g_j=\mu_{j}e^{-r_j}$ ( $\eta_{j},\mu_{j}\in \bC^n,\quad j\in\{1,..,m\}$) belong to $\dom(H^*)$. Their
linear  span is  $2mn$-dimensional subspace  in  $\dom(H^*)$ that
has
 trivial  intersection  with  $\dom(H)$. Since
$\dim(\dom(H^*)/\dom(H)) =2mn$, the  domain  $\dom(H^*)$ takes the form
\eqref{eq3a}.

$(ii)$  The  second  Green  identity  is  verified  similarly  to  3D  case.  
  From   \eqref{eq3a}  it follows that the mapping $\Gamma=(\Gamma_0,\Gamma_1)^\top$
is  surjective.   Namely, let
$(h_0,h_1)^\top\in\cH\oplus\cH$,
 where $h_0=\{h_{0j}\}_{j=1}^m,h_1=\{h_{1j}\}_{j=1}^m$
 are  vectors from $\oplus_{j=1}^m\mathbb{C}^n$.  If $f\in\dom(H^*)$,
  then, by \eqref{eq3a},  $f=f_H+\sum\limits^m_{j=1}\bigl(\xi_{0j}e^{-r_j}\ln(r_j)+\xi_{1j}e^{-r_j}\bigr)$.
  Let us  put
\begin{multline}\label{hat'}
\qquad\qquad\qquad\qquad\qquad\qquad\qquad\xi_0:=\{\xi_{0j}\}_{j=1}^m,\,\,\xi_1:=\{\xi_{1j}\}_{j=1}^m,\,\,\\
E_0:=\left(-e^{-|x_k-x_j|}\ln(|x_k-x_j|+\delta_{kj})\right)_{j,k=1}^m,\quad
E_1:=\left(e^{-|x_k-x_j|}\right)_{k,j=1}^m.
\end{multline}
Therefore if $\xi_0=\tfrac1{2\pi}h_0$ and
$\xi_1={(E_1\otimes
I_n)}^{-1}(h_1+\tfrac1{2\pi}(E_0\otimes I_n)h_0)$,
then $\Gamma_0f=h_0$  and $\Gamma_1f=h_1$.
Thereby, ${\rm (ii)}$ is  proved.

$(iii)$ From  \eqref{eq2} and  \eqref{eq3a} it  follows that any
function  $f\in W^2_2(\bR^2,\bC^n)$ admits the representation
$f=\sum\limits^m_{j=1}\xi_{1j}\,e^{-r_j} + f_H$, where
$\sum\limits_{k=1}^m\xi_{k1}e^{-|x_k-x_j|}= f(x_j)$ which proves
$(iii)$.
\end{proof}
\begin{proposition}\label{pr2a}
Let  $H$ be the minimal Schr\"{o}dinger operator  and let
$\Pi=\{\kH,\Gamma_0,\Gamma_1\}$  be the boundary triplet  for
$H^*$ defined by \eqref{T2}-\eqref{T2'}. Then

$(i)$  the Weyl function $M(\cdot)$ corresponding to the boundary
triplet $\Pi$ has  the form
\begin{equation}\label{W2}
M(z)=\bigoplus\limits_{s=1}^nM_s(z), \quad
M_s(z)=\left(\tfrac1{2\pi}(\psi(1)-
\ln(\tfrac{\sqrt{z}}{2i}))\delta_{jk}+\widetilde{G}_{\sqrt{z}}(x_j-x_{k})\right)_{j,k=1}^m,
\quad z\in\mathbb{C}_+,
 \end{equation}
  where $\psi(1)=\frac{\Gamma'(1)}{\Gamma(1)}$, \quad
      $\widetilde{G}_{\sqrt{z}}(x)=\left\{%
  \begin{array}{ll}
    i/4 H_0^{(1)}(\sqrt{z}|x|), & \hbox{x$\neq$0;} \\
    0, & \hbox{x=0.} \\
\end{array}%
\right.$

and $H^{(1)}_0(\cdot)$ denotes the Hankel function of the first
kind and order 0;

$(ii)$  the corresponding $\gamma(\cdot)$-field  is
\begin{equation}\label{g2}
\gamma(z)\overline{\xi}=\sum\limits^m_{j=1}\xi_j\,
\tfrac{i}4H^{(1)}_0(\sqrt{z}r_j),\quad
\overline{\xi}=\{\xi_j\}_{j=1}^m,\quad \xi_j\in\bC^n,\quad
z\in\bC_+.
\end{equation}
 \end{proposition}
\begin{proof}

 Let  $f_z\in \kN_z,\ z\in\bC_+$. Then,  according  to \cite[chapter II.4]{AGHH88},
\begin{equation*}\label{eq7a}
f_z:=
\sum\limits^m_{j=1}a_j\,\tfrac{i}4H^{(1)}_0(\sqrt{z}r_j),\qquad
a_j\in\bC^n.
\end{equation*}
 It is not  difficult  to  see  that,  by  formulas $(9.01)$ in \cite[Section 2,\S 9]{Olv} and
 (5.03), (5.07)   in  \cite[Section 7,\S5]{Olv},  the function $H^{(1)}_0(z)$ has
  the  following  asymptotic   expansion  at $0$
\begin{equation}\label{eqas}
H^{(1)}_0(z)=1+\tfrac{2i}{\pi}(
\ln(\tfrac{z}2)-\psi(1))+o(z),\quad z\rightarrow 0.
\end{equation}
 Applying $\Gamma_0$ and  $\Gamma_1$ to $f_z$  and taking  into  account  \eqref{eqas}, we
 get
\begin{equation}\label{eq*a}
\Gamma_0f_z=\{a_j\}^m_{j=1},\qquad
 \Gamma_1f_z=\left\{\left(\tfrac{\psi(1)}{2\pi}+\tfrac{i}4-\tfrac{ \ln(\tfrac{\sqrt{z}}2)}{2\pi}\right)a_j
 +\sum\limits_{k\neq j}\tfrac i{4}H_0^{(1)}(\sqrt{z}|x_k-x_j|)a_k\right\}^m_{j=1},\quad
\end{equation}

 Further, combining   \eqref{eq*a} with \eqref{2.3A}, we  get \eqref{W2} and \eqref{g2}.
 \end{proof}

  \subsection{Proper  extensions of the  minimal Schr\"{o}dinger  operator $H$}
 As in previous  section, we  describe proper  extensions of the  minimal operator $H$.
\begin{theorem}\label{pr3'}
Let   $H$  be the minimal  Schr\"{o}dinger  operator, let
$\Pi=\{\kH,\Gamma_0,\Gamma_1\}$  be  the boundary triplet for
$H^*$ defined  by  \eqref{T2}-\eqref{T2'},  and  $M(\cdot)$
the corresponding Weyl function. Assume   also that
$\xi_0,\xi_1, E_0, E_1$  are  defined  by
\eqref{hat'}
%
and  $H_{C,D}$  is a proper extension of $H$. Then  the following
assertions hold.
\newline
$(i)$ Any   proper  extension  $H_{C,D}$ of $H$ is described  as
follows
\begin{equation*}\label{concr'}
\dom(H_{C,D})=\left\{ f \in \dom(H^*):\, D(E_1\otimes
I_n)\xi_1=(2\pi C+D(E_0\otimes I_n))\xi_0
\right\}, \quad C,D\in[\cH].
\end{equation*}
$(ii)$  Extension    $H_{C,D}$  is    self-adjoint  if and only if
\eqref{s-acond} holds.
\newline
$(iii)$  Friedrichs  extension   $H_F$  of  $H$ coincides
 with $H_0$:
\begin{equation*}
\dom(H_F)=\dom(H_0)=W^2_2(\bR^2,\bC^n).
\end{equation*}
$(iv)$  The domain  $\dom(H_K)$ of the Krein extension $H_K$  is
\begin{equation*}
\dom(H_K)=\left\{%
\begin{array}{ll}
    \dom(H_0), & \hbox{m=1;} \\
    \{f\in \dom(H^*):\,\,(\Gamma_0f,\Gamma_1f)^{\top}\in M(0)\}, & \hbox{$m>1$.} \\
\end{array}%
\right.
\end{equation*}
where  
\begin{gather}\label{op1}
\dom(M(0)_{op})=\oplus_{s=1}^n\dom(M_s(0)_{op}),\quad
\dom(M_s(0)_{op})=\left\{\xi=\{\xi_j\}_{j=1}^m\in
\mathbb{C}^m:\,\sum\limits_{j=1}^m\xi_j=0 \right\},\\\label{mul}
\mul(M(0))=\oplus_{s=1}^n\Span\{e_{mul}\},\quad
e_{mul}=\{e_j\}_{j=1}^m=\{1\}_{j=1}^m.
\end{gather}

 $(v)$ Krein  formula for  canonical  resolvents  takes
the form
 \begin{equation*}\label{Kre2}
 R_z(H_{C,D}) = R_z(H_0) + \gga(z)\bigl(C
 -DM(z)\bigr)^{-1}D\gga(\overline{z})^*,\,\,
 z\in\rho(H_{C,D})\setminus\bR_+,
 \end{equation*}
where $\gamma(\cdot)$-field is  defined  by \eqref{g2} and
$R_z(H_0)$ is an integral operator with  the kernel
$G_{\sqrt{z}}(x,x')=i/4 H_0^{(1)}(\sqrt{z}|x-x'|)\otimes I_n$ .
\end{theorem}
\begin{proof}
$(i)$ and $(ii)$  follow from the  representation \eqref{ext repr}.

 $(iii)$
 From  the  asymptotic representation (see, for   instance,  formula (4.03) in \cite[Section 7,\S
 4]{Olv})
 \begin{equation*}
H_0^{(1)}(z)\sim\sqrt{\tfrac2{\pi z}}e^{i(z-\pi/4)}, \quad
|z|\rightarrow \infty,
 \end{equation*}
 it  easily  follows that  $\lim\limits_{x\downarrow-\infty}(M(x)f,f)=-\infty,\quad f\in\kH\setminus\{0\}.$
 Thus, by    Proposition  \ref{prkf}$(iv)$, $H_F=H_0$.
   \newline
$(iv)$  In the case $m=1$,  the  Weyl  function  has the form
$M(z)=(\psi(1)- \ln(\tfrac{\sqrt{z}}{2i}))I_n.$ The  latter yields
\begin{equation*}
\lim\limits_{x\downarrow-\infty}(M(x)f,f)=-\infty,\qquad
\lim\limits_{x\uparrow0}(M(x)f,f)=+\infty,\quad
f\in\kH\setminus\{0\}.
\end{equation*}
By Proposition  \ref{prkf}$(iv)$, $H_F=H_K=H_0$.
 Furthermore, from the equality  $H_K=H_F$  it  follows that  operator  $H$  has no other nonnegative  self-adjoint
 extensions  (see \cite{Kre47}).


  Consider  the case  $m>1$.
 For simplicity suppose
that $n=1$.  Let  $\xi=\{\xi_j\}_{j=1}^m\in\mathbb{C}^m$. Using
asymptotic   expansion \eqref{eqas}, 
we   get
\begin{multline}\label{qfop}
(M(z)\xi,\xi)\sim \tfrac1{2\pi}(\psi(1)-
\ln(\tfrac{\sqrt{z}}{2i}))\sum\limits_{j=1}^m|\xi_j|^2+\sum\limits_{k\neq
j}\tfrac1{2\pi}\left(\psi(1)-
\ln(\tfrac{\sqrt{z}}{2i})-\ln(|x_k-x_j|)\right)\xi_j\overline{\xi_k}=\\=\tfrac1{2\pi}(\psi(1)-
\ln(\tfrac{\sqrt{z}}{2i}))\left(|\sum\limits_{j=1}^m\xi_j|\right)^2-\tfrac1{2\pi}\sum\limits_{k\neq
j}\ln(|x_k-x_j|)\xi_j\overline{\xi_k},\quad z\rightarrow 0.
\end{multline}
From \eqref{qfop} it  easily  follows   that  limit
$\lim\limits_{x\uparrow 0}(M\xi,\xi)$  is finite  if  and  only if
$\sum\limits_{j=1}^m\xi_j=0$.  Thus,  the  domain of the  operator  part  $M(0)_{op}$  is  described by  \eqref{op1}
Finally,  \eqref{mul}  takes
place since  $\mul(M(0))$  and $\dom(M(0)_{op})$   are orthogonal.

   Applying  Proposition  \ref{prkf}$(ii)$
completes  the proof of $(iv).$

  Combining  \eqref{krein} with the  formula  for  the
 kernel of  $(H_0-z)^{-1}$ (see \cite[chapter I.5]{AGHH88}), we
 obtain $(v)$.
\end{proof}
As in the  case of 3D Schr\"{o}dinger operator, only  certain
$m$-parametric family $H_{\alpha,X}^{(2)}$ associated in
$L^2(\bR^2)$ with the
 differential  expression  \eqref{eq0} is   described in \cite[chapter II.1, Theorem 4.1]{AGHH88}.
\begin{proposition}
Let  $\Pi$ be  the  boundary triplet for $H^*$ defined by
  \eqref{T2}-\eqref{T2'}.
  Then  the domain of   $H_{\alpha,X}^{(2)}$   has  the following representation
 \begin{equation*}\label{diag2}
\dom(H^{(2)}_{\alpha,X})=\dom(H^*)\upharpoonright\ker(\Gamma_1-B_\alpha\Gamma_0),
\quad B_\alpha=\diag(\alpha_1,.., \alpha_m),
\,\alpha_k\in\mathbb{R},\, k\in\{1,..,m\}.
 \end{equation*}
 \end{proposition}

 Note  that in the
case   $d=2$ it  makes  certain  difficulty to describe
nonnegative  self-adjoint  extensions of $H$ since  $M(0)$ appears
to  be  the relation  with   nontrivial  multivalued  part. We may
overcome this  by  considering  the following  intermediate
extension  of $H$.
\begin{equation*}\label{ext1}
\widetilde{H}:=H^*\upharpoonright\dom(\widetilde{H}),\quad
\dom(\widetilde{H})=\dom(H_F)\cap\dom(H_K).
\end{equation*}
As   above,   assume that $n=1$.  It  is easily  seen that  
\begin{equation*}\label{ext2}
\dom(\widetilde{H})=\left\{ f=
c\sum\limits^m_{j=1}\widetilde{\xi}_j\,e^{-r_j} + f_H\ :\,\widetilde{\xi}=\{\widetilde{\xi}_j\}_{j=1}^{m}=E^{-1}_1e_{mul},\,\,\,\,
c\in\bC,\quad f_H\in\dom(H)\right\}, \quad
\end{equation*}
where  $E_1$ is   defined by  \eqref{hat'}.

 According
 to \cite{Rof85}, we  have
\begin{equation*}
  \cH=\cH_1\oplus \cH_2, \quad \cH_1=\dom(M(0)_{op})  \text{\quad and\quad } \cH_2=\mul(M(0)).
  \end{equation*}  Let  $\pi_j,\,j\in \{1,2\}$ denote the  orthogonal
 projectors  onto  $\cH_j$ . Then the  Weyl function $M(\cdot)$  defined  by \eqref{W2}  admits the  representation
 $M(\cdot)=(M_{kj}(\cdot))_{k,j=1}^2$ with  $M_{k,j}(\cdot)=\pi_k M(\cdot)\upharpoonright\cH_j,\quad
 k,j\in\{1,2\}$.
 One  may simply   verify  that
 \begin{equation*}
 \widetilde{H}=H_1:=H^*\upharpoonright\{f\in \dom(H^*):\,
 \Gamma_0f=\pi_1\Gamma_1f=0\},
 \end{equation*}
with  $\Gamma_0, \Gamma_1$  defined   by \eqref{T2}-\eqref{T2'}.
   From  \cite[Proposition 4.1]{dhms} it  follows  that
 $\widetilde{H}$  is  closed  symmetric operator  in  $L^2(\bR^2)$
  with  deficiency  indices
  $n_{\pm}(\widetilde{H})=\dim(\cH_1)=m-1$.  Proposition
  4.1$(ii)$  in  \cite{dhms} also  yields   that
 $H_1^*=\widetilde{H}^*=H^*\upharpoonright\{f\in \dom(H^*):\,
 \pi_2\Gamma_0f=0\}$,    and  boundary   triplet
$\widetilde{\Pi}=\{\widetilde{\cH},\widetilde{\Gamma}_0,\widetilde{\Gamma}_1\}$
for $\widetilde{H}^*$  might  be defined as   follows
\begin{equation*}
\widetilde{\cH}=\cH_1,\quad
\widetilde{\Gamma}_0=\Gamma_0\upharpoonright\dom(\widetilde{H}^*),\quad
\widetilde{\Gamma}_1=\pi_1\Gamma_1\upharpoonright\dom(\widetilde{H}^*).
\end{equation*}
Moreover,  the Weyl function  $\widetilde{M}(\cdot)$ corresponding
to the boundary  triplet $\widetilde{\Pi}$   are  given by
$\widetilde{M}(\cdot)=M_{11}(\cdot)$  and  the equality
$\widetilde{M}(0)=M(0)_{op}$ is   satisfied.

%
%
\begin{proposition}
Let   $H$   and    $M(0)_{op}$  be  as above and let $H'$ be a
non-negative self-adjoint  extension of $H$. Then

(i)
There exist pairs  $C,D\in[\cH]$  and
$\widetilde{C},\widetilde{D}\in[\widetilde{\cH}]$  satisfying
\eqref{s-acond} and such that
%
%
   \begin{equation*}\label{corr}
H'= H_{C,D}=H^*\upharpoonright\ker(D\Gamma_1-C\Gamma_0)=
\widetilde{H}^*\upharpoonright\ker(\widetilde{D}\widetilde{\Gamma}_1-\widetilde{C}\widetilde{\Gamma}_0)
=:\widetilde{H}_{\widetilde{C},\widetilde{D}}.
       \end{equation*}

(ii)  The extension $H_{C,D} = H_{C,D}^*$ is nonnegative if and
only if
    \begin{equation*}
(\widetilde{C}\widetilde{D}^*-\widetilde{D}M(0)_{op}\widetilde{D}^*h,h)\geq
0, \qquad  h\in \dom(M(0)_{op})\setminus\{0\}.
\end{equation*}
   \end{proposition}
       \begin{remark}
$(i)$ The uniqueness of  nonnegative  self-adjoint  extension of
2D operator $H$,  in the  case $n=m=1$, was    established in
\cite{Ges01} and  \cite{Ada07}.

$(ii)$
 In \cite{Ada07},  V. Adamyan    noted  that, in  the  case $m>1$  and $n=1$,  the  operator  $H$
has  non-unique  nonnegative  self-adjoint extension.
 \end{remark}

\subsection{Spectrum  of the  self-adjoint  extensions of  the  minimal Schr\"{o}dinger  operator  and scattering  matrix }
Point spectrum  of  the self-adjoint  extensions of    $H$ is described
in the following  theorem.
\begin{theorem}\label{spec2}
Let  $H$ be the operator defined by  \eqref{eq2}, let  $\Pi$ be
the boundary triplet for  $H^*$ defined by \eqref{T2}-\eqref{T2'},
and let $M(\cdot)$ be  the corresponding Weyl function. Assume
also  that $H_\Theta$ is  a self-adjoint   extension of $H$.
  Then  point  spectrum of
   the self-adjoin  extension  $H_\Theta$  consists of  at most $nm$
      negative eigenvalues (counting multiplicities). Moreover,
$z\in\sigma_p(H_\Theta)\cap\bR_-$
 if and only  if $0\in\sigma_p(\Theta-M(z))$, i.e.,
 \begin{equation*}
 z\in\sigma_p(H_\Theta)\cap\bR_- \Leftrightarrow 0\in\sigma_p(C-DM(z)).
\end{equation*}
The corresponding   eigenfunction $\psi_z$ has the  form
\begin{equation*}\label{eigen'}
\psi_z=\sum\limits_{j=1}^mc_j\tfrac{i}4H^{(1)}_0(\sqrt{z}r_j),
\end{equation*}
where $(c_1,.., c_m)^\top$ is eigenvector  of  the   relation
$\Theta-M(z)$ corresponding  to zero eigenvalue.
\end{theorem}





 As in the case of 3D  Schr\"{o}dinger  operator,
2D Schr\"{o}dinger operator $H$ is not  simple. Arguing  as above,
we obtain
\begin{theorem}
Scattering  matrix $\{\widehat{S}_\gT(z)\}_{z \in \bR_+}$ of the
scattering  system $\{\widehat{H}_\gT,\widehat{H}_0\}$ has the
form
\begin{gather*}
 \widehat{S}_\gT(x) = I_{nm}
+2i\sqrt{J(x)}\bigl(\gT -I_n\otimes\left(\tfrac1{2\pi}(\psi(1)-
\ln(\tfrac{\sqrt{x}}{2i}))\delta_{jk}+\widetilde{G}_{\sqrt{z}}(x_j-x_{k})\right)_{j,k=1}^m\bigr)^{-1}\sqrt{J(x)},\,\,
\\
J(x)=I_{n}\otimes\left(\tfrac1{4}J_0(\sqrt{x}|x_j-x_{k}|)\right)_{j,k=1}^m,\quad
x \in \bR_+,
\end{gather*}
where $J_0(\cdot)$  denotes  Bessel function.
\end{theorem}
\ack{The author  thanks  M.M. Malamud for posing the problem and
permanent  attention to the  work. The  author also acknowledges
the  referees   for carefully reading of the preliminary version
of the manuscript and constructive remarks.}

\quad
\\
Nataly Goloshchapova,\\
\emph{Institute of Mathematics and Statistics, University of S\~ao Paulo},\\
\emph{ Rua do Mat\~ao, 1010},\\
\emph{S\~ao Paulo, 05508-090, Brazil}\\
 \emph{e-mail:} nataliia@ime.usp.br
\end{document}